\documentclass[10pt]{article}
\usepackage{pifont}
\usepackage{amsfonts}
\usepackage{amsmath,amsthm,amssymb}
\usepackage{amsmath, epsfig}
\usepackage{latexsym}
\usepackage{amssymb}
\usepackage{mathrsfs}
\usepackage{amscd}
\usepackage{xypic}
\usepackage[all]{xy}
\usepackage{dsfont}
\usepackage{color}


 
\newtheorem{thm}{Theorem}[section]

\newtheorem{lem}[thm]{Lemma}

\newtheorem{cor}[thm]{Corollary}

\newtheorem{prop}[thm]{Proposition}
\newtheorem{rmk}[thm]{Remark}

\begin{document}
\title{\bf Blocks with the hyperfocal subgroup
$\mathbb Z_{2^n}\times \mathbb Z_{2^n}$}

\author{Xueqin Hu and Yuanyang Zhou}

\date{\small
School of Mathematics and Statistics, Central China Normal University,
Wuhan 430079, China}
\maketitle

\begin{abstract}

In this paper, we calculate the numbers of irreducible ordinary characters and irreducible Brauer characters in a block of a finite group $G$,
whose associated fusion system over a $2$-subgroup $P$ of $G$
(which is a defect group of the block) has hyperfocal subgroup
$\mathbb Z_{2^n}\times \mathbb Z_{2^n}$ for a positive integer number $n$,
when the block is controlled by the normalizer $N_G(P)$ and the hyperfocal subgroup is contained in the center of $P$,
or when the block is not controlled by $N_G(P)$ and
the hyperfocal subgroup is contained in the center of the unique essential subgroup in the fusion system
and has order at most $16$.
In particular, Alperin's weight conjecture holds in the considered cases.

\end{abstract}
\section{Introduction}
\bigskip\quad\, Let $p$ be a prime, and
$\mathcal{O}$ a complete discrete valuation ring with
an algebraically closed residue field $k$ of characteristic $p$
and with a fraction field $\mathcal{K}$ of characteristic 0.
We further assume that $\mathcal{K}$ is big enough for finite groups discussed below.
Let $G$ be a finite group, and $b$ a block (idempotent) of $G$ over $\mathcal{O}$ with defect group $P$.
Let $(P, \,b_P)$ be a maximal $b$-Brauer pair.
Denote by $Q$ the subgroup of $P$ generated by
the subsets $[U,\,x]$,
where $(U,\,b_U)$ runs on the set of $b$-Brauer pairs
such that $(U,\,b_U)\subseteq(P,\,b_P)$
and $x$ on the set of $p^\prime$-elements of $N_G(U,\,b_U)$
and $[U,\,x]$ denotes the set of commutators $uxu^{-1}x^{-1}$ for $u\in U$.
The subgroup $Q$ is called the hyperfocal subgroup of
$P$ with respect to $(P,\, b_P)$ (see \cite{P00}).
Let $c$ be the Brauer correspondent of $b$ in $N_G(Q)$.
Rouquier conjectures that the block algebras $\mathcal O Gb$ and
$\mathcal O N_G(Q)c$ are basically Rickard equivalent when $Q$ is abelian (see \cite{R98}).

Denote by $l(b)$ \big(resp. $k(b)$\big) the number of irreducible Brauer
(resp. ordinary) characters belonging to the block $b$.
Rouquier's conjecture implies that $l(b)$ and $k(b)$ should be equal to $l(c)$ and $k(c)$ respectively when $Q$ is abelian.
Recently, Watanabe proved that
$l(b)$ and $k(b)$ are equal to $l(c)$ and $k(c)$ respectively
when $Q$ is cyclic (see \cite[Theorem 1]{W14}).
In addition, she proved that the block $b$ is controlled by $N_G(P)$ and
that $l(b)$ and $k(b)$ are equal to $l(b_0)$ and $k(b_0)$ respectively,
where $b_0$ denotes the Brauer correspondent of $b$ in $N_G(P)$. The block $b$ is controlled by a subgroup $H$ of $G$ if any morphism in the Brauer category ${\cal F}_{(P, \,b_P)}(G, \,b)$ is induced by the conjugation of some element in $H$ (see \cite[\S 49]{T95}).

In this paper, we investigate the numbers $l(b)$ and $k(b)$
when $Q$ is $\mathbb Z_{2^n}\times \mathbb Z_{2^n}$,
where $\mathbb Z_{2^n}$ denotes a cyclic group of order $2^n$.
In this situation, the block $b$ is not necessarily controlled by $N_G(P)$ in general.
When the block $b$ is not controlled by $N_G(P)$,
by induction the investigation is reduced to the investigation of the numbers $l(b)$ and $k(b)$ when $P$ is the wreath product $Q\wr \mathbb Z_2$.
The latter investigation is difficult (see \cite{K80} for details),
but we get the numbers $l(b)$ and $k(b)$ when $n\leq 2$.
Our main theorem is stated as the following.

\begin{thm}\label{MT}
Keep the notation above and assume that $Q$ is $\mathbb Z_{2^n}\times \mathbb Z_{2^n}$ for a positive integer number $n$.

\noindent{\bf (i)} Assume that the block $b$ is controlled by $N_G(P)$ and that $Q$ is contained in the center of $P$.
Then we have $l(b)=l(c)=l(b_0)=3$ and $k(b)=k(c)=k(b_0)$,
where $b_0$ is the Brauer correspondent of the block $b$ in $N_G(P)$.

\noindent{\bf (ii)} Assume that the block $b$ is not controlled by $N_G(P)$.
Then the Brauer category ${\cal F}_{(P, \,b_P)}(G, \,b)$ of the block $b$,
whose objects are all $b$-Brauer pairs contained in a fixed maximal $b$-Brauer pair
$(P, \,b_P)$, has a unique essential object $(S,\, b_S)$.
Further assume that $Q$ is contained in the center of $S$ and
that $|Q|$ is less than $16$.
Then we have $l(b)=l(c)=2$ and $k(b)=k(c)$.
\end{thm}

The special case $n=1$ of Theorem \ref{MT} has been done by \cite{T}.


\begin{rmk}\rm
There are examples where the hypotheses of
Theorem \ref{MT} are not satisfied.
Assume that $Q=\mathbb{Z}_4\times\mathbb{Z}_4$.
The automorphism group $\mathrm{Aut}(Q)$ of $Q$
has a cyclic subgroup of order $6$.
Set $L=Q\rtimes\mathbb{Z}_6$,
where $\mathbb{Z}_6$ is a cyclic group of order $6$.
Take any $2^\prime$-group $H$.
Denote by $G$ the wreath product $H\wr L$.
Let $b$ be the principal block of $G$
and $P$ a Sylow $2$-sugbroup containing $Q$.
The block $b$ is controlled by $N_G(P)$ and $Q$ is a hyperfocal subgroup, but $Q$ is not a central subgroup of $P$.
On the other hand,
$\mathrm{Aut}(Q)$ contains a subgroup
$\mathbb{Z}_2\times\mathrm{S}_3$,
where $\mathrm{S}_3$ is the symmetric group of degree $3$.
Set $K=Q\rtimes(\mathbb{Z}_2\times\mathrm{S}_3)$ and $\tilde{G}=H\wr K$.
Let $\tilde{b}$ be the principal block of $\tilde{G}$
and $\tilde{P}$ a Sylow $2$-sugbroup containing $Q$.
The subgroup $Q\rtimes\mathbb{Z}_2$ is the unique essential subgroup $S$ of $\tilde{P}$,
$Q$ is a hyperfocal subgroup
and the block $\tilde{b}$ is controlled by $N_{\tilde{G}}(S)$.
But $Q$ is not in the center of $S$.
\end{rmk}

Finally we claim that throughout the rest of the paper,
the notation in this section will always be kept and $Q$ is
$\mathbb Z_{2^n}\times \mathbb Z_{2^n}$ for a positive integer number $n$.
In particular, $p=2$.

\section{The local structure of the block $b$}

\bigskip\quad\, In this section, we investigate the Brauer category
$\mathcal{F}_{(P,\, b_P)}(G,\, b)$.
Let $Q_0$ be the subgroup $\mathbb Z_2\times \mathbb Z_2$ of $Q$.
The following lemma is trivial.

\begin{lem}\label{Aut(Q)}
The automorphism group $\mathrm{Aut}(Q)$ of $Q$ is a $\{2,3\}$-group
 and the automorphism group $\mathrm{Aut}(Q_0)$ of $Q_0$ is isomorphic to
 ${\rm S}_3$, the symmetric group of degree $3$.
 If $\theta\in\mathrm{Aut}(Q)$ has order three,
 then
 $\theta$ transitively permutes three involutions in $Q_0$
 and generates a Sylow $3$-subgroup of $\mathrm{Aut}(Q)$.
\end{lem}

For any subgroup $T$ of $P$,
we denote by $(T,\,b_T)$ the $b$-Brauer pair contained in $(P,\, b_P)$.
We denote by $N_G(T,\,b_T)$ the normalizer of $(T,\,b_T)$
under the $G$-conjugation.
The block $b_T$ of $C_G(T)$ is also a block of $N_G(T,\,b_T)$.
Since $P$ is contained in $N_G(Q,\, b_Q)$ and $P$ is a defect group of $b$,
$P$ has to be a defect group of the block $b_Q$ of $N_G(Q,\, b_Q)$.
Similarly, $P$ is also a defect group of the block
$b_{Q_0}$ of $N_G(Q_0,\, b_{Q_0})$.
Set $P_0=C_P(Q_0)$.
Then $P_0$ is a defect group of the block $b_{Q_0}$ of $C_G(Q_0)$.

\begin{lem}\label{nil}
The block $b_{Q_0}$ of $C_G(Q_0)$ is nilpotent.
\end{lem}

\begin{proof}
Set $L=N_G(Q,\,b_Q)$ and $e=b_Q$.
Regarding $e$ as a block of $L$,
the $b$-Brauer pair $(P,\,b_P)$ is a maximal $e$-Brauer pair.
For any subgroup $T$ of $P$,
denote by $(T, \,e_T)$ the $e$-Brauer pair contained in $(P,\, b_P)$.
By \cite[Theorem 2]{W14}, the correspondence $(T,\, b_T)\mapsto (T, \,e_T)$ induces an isomorphism between the Brauer categories $\mathcal{F}_{(P,\, b_P)}(G,\, b)$ and $\mathcal{F}_{(P,\, b_P)}(L,\, e)$.
Set $H=C_G(Q_0)$, $K=C_L(Q_0)$, $d=b_{Q_0}$ and $f=e_{Q_0}$.
The $b$-Brauer pair $(P_0,\, b_{P_0})$ is a maximal $d$-Brauer pair and
the $e$-Brauer pair $(P_0,\, e_{P_0})$ is a maximal $f$-Brauer pair.
For any subgroup $T$ of $P_0$,
denote by $(T, \,d_T)$ the $d$-Brauer pair contained in $(P_0,\, b_{P_0})$ and
by $(T, \,f_T)$ the $f$-Brauer pair contained in $(P_0,\, e_{P_0})$.
By \cite[Corollary 3.6]{P01}, the correspondence $(T, \,d_T)\mapsto (T,\, f_T)$ induces an isomorphism between the Brauer categories
$\mathcal{F}_{(P_0,\, b_{P_0})}(H,\, d)$ and $\mathcal{F}_{(P_0,\, e_{P_0})}(K,\, f)$.
So in order to prove that the block $d$ is nilpotent,
we may assume without loss of generality that $G$ is equal to $L$.

Since $G=L$, the group $C_G(Q_0)/C_G(Q)$ is isomorphic to a subgroup of $\mathrm{Aut}(Q)$ consisting of automorphisms stabilizing all elements in $Q_0$. By Lemma \ref{Aut(Q)}, the group $C_G(Q_0)/C_G(Q)$ is a $2$-group.
On the other hand, since $P$ is a defect group of the block $b$,
the quotient group $PC_G(Q)/C_G(Q)$ is a Sylow $2$-subgroup of $G/C_G(Q)$.
The quotient group $C_G(Q_0)/C_G(Q)$ is normal in $G/C_G(Q)$.
So we have $C_G(Q_0)/C_G(Q)\subset PC_G(Q)/C_G(Q)$ and $C_G(Q_0)=P_0C_G(Q)$.
Since the block $b_{Q_0}$ of $C_G(Q_0)$ covers the block $b_Q$ of $C_G(Q)$,
by \cite[Proposition 4.2]{P00}
and \cite[Proposition 6.5]{KP}
$b_{Q_0}$ is nilpotent.
\end{proof}

\begin{prop}\label{fusion system}
 {\bf (a)} Assume that the block $b$ is controlled by $N_G(P)$.
 Then $Q_0\leq Z(P)$, the center of $P$.

 \noindent{\bf (b)} Assume that the block $b$ is not controlled by $N_G(P)$.
Then $\mathcal{F}_{(P,\, b_P)}(G,\,b)$ has only one essential object $(S, \,b_S)$.
Moreover, $S=C_P(Q_0)$, $|P:S|=2$ and the block $b$ is controlled by $N_G(S,\,b_S)$.
\end{prop}

\begin{proof}
In order to prove the proposition, by \cite[Theorem 2]{W14} we may assume that $G$ is equal to $N_G(Q,\, b_Q)$.
Then $b_Q=b$ and $Q_0$ is normal in $G$.
The quotient group $P/P_0$ is isomorphic to a $2$-subgroup of ${\rm S}_3$.
By Lemma \ref{Aut(Q)}, the index of $P_0$ in $P$ is equal to $1$ or $2$.

We firstly prove the statement {\bf (a)}.
Since the pair $(P_0,\, b_{P_0})$ is a maximal $b_{Q_0}$-Brauer pair,
the block $b_{P_0}$ of $C_G(P_0)$ is nilpotent.
Note that $b_{P_0}$ is also a block of $PC_G(P_0)$ and
that $(P,\, b_P)$ is a maximal $b_{P_0}$-Brauer pair.
By \cite[Proposition 6.5]{KP}, the block $b_{P_0}$ of $PC_G(P_0)$ is nilpotent.
So we have
$N_G(P,\,b_P)\cap PC_G(P_0)=N_{PC_G(P_0)}(P,\,b_P)=PC_G(P)$
and then the inclusion $N_G(P,\,b_P)\subset N_G(P_0,\,b_{P_0})$ induces an injective group homomorphism
$N_G(P,\,b_P)/PC_G(P)\rightarrow N_G(P_0,\,b_{P_0})/PC_G(P_0)$.
The group $N_G(P,\,b_P)/PC_G(P)$ is a nontrivial group.
Otherwise, $N_G(P,\,b_P)=PC_G(P)$.
Since we are assuming in the statement {\bf (a)} that
the block $b$ is controlled by $N_G(P)$, the block $b$ is nilpotent.
That contradicts with the hyperfocal subgroup $Q$ of the block $b$ being nontrivial. Therefore the quotient group $N_G(P_0,\,b_{P_0})/PC_G(P_0)$ is nontrivial; moreover it has odd order since it is well known that the quotient group
$N_G(P,\,b_P)/PC_G(P)$ has order coprime to $2$.

By Lemma \ref{Aut(Q)}, the block $b_{Q_0}$ of $C_G(Q_0)$ is nilpotent.
So $N_G(P_0,\,b_{P_0})\cap C_G(Q_0)
    =N_{C_G(Q_0)}(P_0,\,b_{P_0})
    =P_0C_G(P_0)$ and then the inclusion $N_G(P_0,\,b_{P_0})\subset G$
induces an injective homomorphism
$N_G(P_0,\,b_{P_0})/P_0C_G(P_0)\rightarrow G/C_G(Q_0)$.
By Lemma \ref{Aut(Q)} again, the quotient group $G/C_G(Q_0)$ is isomorphic to a subgroup of ${\rm S}_3$, Sylow $3$-subgroups of
$N_G(P_0,\,b_{P_0})/P_0C_G(P_0)$ are nontrival.

Now we assume that $P\neq P_0$.
Then Sylow $2$-subgroups of $N_G(P_0,\,b_{P_0})/P_0C_G(P_0)$ are nontrivial. Therefore $N_G(P_0,\,b_{P_0})/P_0C_G(P_0)$ has to be isomorphic to ${\rm S}_3$. This shows that $(P_0,\,b_{P_0})$ is essential in $\mathcal{F}_{(P,\,b_P)}(G,\,b)$.
But since we are assuming in the statement {\bf (a)} that
the block $b$ is controlled by $N_G(P)$,
$\mathcal{F}_{(P,\,b_P)}(G,\,b)$ has no essential object!
That produces a contradiction.
So we have $P=P_0$ and the statement {\bf (a)} is proved.

Now we prove the statement {\bf (b)}.
In this case, $\mathcal{F}_{(P,\,b_P)}(G,\,b)$ has essential objects.
Let $(S,\, b_S)$ be an essential object in $\mathcal{F}_{(P,\,b_P)}(G,\,b)$.
Then $S$ contains $Q$ since $Q$ is normal in $G$.
So the pair $(S,\, b_S)$ is also a Brauer pair of the block $b_{Q_0}$ of $PC_G(Q_0)$.
By Lemma \ref{nil}, the block $b_{Q_0}$ of $C_G(Q_0)$ is nilpotent,
and by \cite[Proposition 6.5]{KP}
the block $b_{Q_0}$ of $PC_G(Q_0)$ is nilpotent too.
So we have $N_{PC_G(Q_0)}(S,\,b_S)=PC_G(S)$,
and then $N_G(S,\,b_S)\cap C_G(Q_0)=N_G(S,\,b_S)\cap PC_G(Q_0)\cap C_G(Q_0)=PC_G(S)\cap C_G(Q_0)=P_0C_G(S)$.
Since $Q_0$ is normal in $G$, $P_0C_G(S)$ is normal in $N_G(S,\,b_S)$ and
thus the image of $P_0C_G(S)$ in $N_G(S,\,b_S)/SC_G(S)$ is a
normal $2$-subgroup.
Since $(S,\, b_S)$ is essential in $\mathcal{F}_{(P,\,b_P)}(G,\,b)$,
we have $P_0SC_G(S)=SC_G(S)$.
It is clear that
$(S,\,b_S)$ is selfcentralizing (see \cite[\S41]{T95}).
Then $SP_0=S$.
Since $S\neq P$ and $|P:P_0|=2$, $S$ has to be $P_0$.
In particular, $\mathcal{F}_{(P,\,b_P)}(G,\,b)$ has the unique essential object
$(S,\,b_S)$.
Since $N_G(P,\, b_P)\subset N_G(S,\, b_S)$, by Alperin's fusion theorem
we have
$\mathcal{F}_{(P,\,b_P)}(G,\,b)=\mathcal{F}_{(P,\,b_P)}\big(N_G(S,\,b_S),\,e\big)$,
where $e$ is the block $b_S$ of $N_G(S,\,b_S)$.
\end{proof}

The canonical homomorphism
$N_G(P,\,b_P)/C_G(P)\rightarrow N_G(P,\,b_P)/P C_G(P)$ has a section homomorphism. In particular, $N_G(P,\,b_P)/C_G(P)$ has a $2$-complement $E/C_G(P)$.
When the block $b$ is controlled by $N_G(P)$,
by the proof of the statement {\bf (a)} above,
$N_G(P,\,b_P)/PC_G(P)$ has order $3$ and so does $E/C_G(P)$.

We continue to use the notation in Proposition \ref{fusion system} ({\bf b}) and
assume that the block $b$ is not controlled by $N_G(P)$.
We claim that $N_G(S,\,b_S)/SC_G(S)$ is isomorphic to ${\rm S}_3$.
Indeed, by \cite[Theorem 2]{W14} we may assume that $G$ is equal to $N_G(Q,\, b_Q)$. In the last paragraph of the proof of Proposition \ref{fusion system} ({\bf b}),
we already prove that the intersection
$N_G(S,\,b_S)\cap C_G(Q_0)$ is equal to $SC_G(S)$.
So the inclusion $N_G(S,\,b_S)\subset G$ induces an injective group homomorphism $N_G(S,\,b_S)/SC_G(S)\rightarrow G/C_G(Q_0)$.
Since $(S,\,b_S)$ is essential and $G/C_G(Q_0)$ is isomorphic to
a subgroup of ${\rm S}_3$,
the injective group homomorphism has to be an isomorphism.
The claim is done.

So $N_G(S,\,b_S)$ has a normal subgroup $A$ containing $SC_G(S)$
such that $A/SC_G(S)$ has order three.
Let $E_S/C_G(S)$ be a $2$-complement of $A/C_G(S)$.

\begin{lem}\label{P=QC_P(E)}
Keep the notation as above.

\noindent{\bf (i)} If
the block $b$ is controlled by $N_G(P)$,
then $E\cap C_G(Q)=C_G(P)$,
$P=Q\rtimes C_P(E)$ and $Q=[Q,\,E]$.

\noindent{\bf (ii)} If
the block $b$ is not controlled by $N_G(P)$,
then $E_S\cap C_G(Q)=C_G(P)$,
$S=Q\rtimes C_S(E_S)$ and $Q=[Q,\,E_S]$.
\end{lem}

\begin{proof}
We firstly prove the statement {\bf (i)}.
The block $b_Q$ of $C_G(Q)$ is nilpotent (see \cite[Proposition 4.2]{P00}),
and $b_Q$ as a block of $P C_G(Q)$ is nilpotent too (see \cite[Proposition 6.5]{KP}). Since $(P,\, b_P)$ is a maximal Brauer pair of the block $b_Q$ of $PC_G(Q)$,
we have $N_G(P,\,b_P)\cap PC_G(Q)=P C_G(P)$.
So
\begin{center}
$N_G(P,\,b_P)\cap C_G(Q)=C_P(Q)C_G(P)$ and
$E\cap C_G(Q)=E\cap C_P(Q)C_G(P)=C_G(P)$
\end{center}
and the inclusion $E\subset N_G(Q,\, b_Q)$ induces an injective group homomorphism $E/C_G(P)\rightarrow N_G(Q,\, b_Q)/C_G(Q)$.
We identify $E/C_G(P)$ as the image of the homomorphism.
Since the block $b$ is controlled by $N_G(P)$, $E/C_G(P)$ has order $3$.
By Lemma \ref{Aut(Q)}, $E/C_G(P)$ acts freely on $Q-\{1\}$.
Thus we have $C_Q(E)=1$ and $Q=[Q, \,E]$.
Since $[P,\, E]= Q$, the $E$-conjugation induces a trivial action of
$E/C_G(P)$ on the quotient group $P/Q$.
By a Glauberman Lemma,
every coset of $Q$ in $P$ has an element fixed by $E/C_G(P)$.
This implies that $P=QC_P(E)$.
So we have $P=Q\rtimes C_P(E)$.

The proof of the statement ({\bf ii}) is similar to that of the statement ({\bf i}).
\end{proof}

Let $R$ be a $p$-subgroup of $P$ and $K$ a subgroup of ${\rm Aut}(R)$.
The $N_G(R,\,b_R)$-conjugation induces a group homomorphism
$N_G(R,\,b_R)\rightarrow {\rm Aut}(R)$.
We denote by $N_G^K(R)$ the inverse image of $K$ in $N_G(R,\,b_R)$,
and set $N_P^K(R)=P\cap N_G^K(R)$.
Clearly, $b_R$ is a block of $N_G^K(R)$.
The $p$-subgroup $R$ is said to be fully $K$-normalized in
$\mathcal{F}_{(P,\,b_P)}(G,\,b)$
if $N_P^K(R)$ is a defect group of the block $b_R$ of $N_G^K(R)$
(see \cite[2.6 and Proposition 3.5]{P01}).

\begin{lem}\label{hyperfocal subgroup of local subgroup}
Keep the notation as above and assume that $R$ is fully $K$-normalized in
$\mathcal{F}_{(P,\,b_P)}(G,\,b)$.
Then $Q$ contains the hyperfocal subgroup of the block $b_R$ of $N_G^K(R)$ with respect to a suitable maximal Brauer pair $(N_P^K(R),\, g^K)$.
\end{lem}

\begin{proof}
Set $N=RN_P^K(R)$.
Clearly, $C_G(N)$ is normal in $C_{N_G^K(R)}\big(N_P^K(R)\big)$,
which is contained in $N_G(N)$,
the block $b_N$ of $C_G(N)$ determines the unique block $b_N^K$ of
$C_{N_G^K(R)}\big(N_P^K(R)\big)$ such that $b_Nb_N^K\neq 0$,
and the pair $\big(N_P^K(R),\, b_N^K\big)$ is a Brauer pair of
the block $b_R$ of $N_G^K(R)$.
Since $R$ is fully $K$-normalized in $\mathcal{F}_{(P,\,b_P)}(G,\,b)$,
the pair $\big(N_P^K(R),\, b_N^K\big)$ is a maximal Brauer pair of
the block $b_R$ of $N_G^K(R)$.
Set ${\cal F}=\mathcal{F}_{(P,\,b_P)}(G,\,b)$.
We denote by $N_{\cal F}^K(R)$ the $K$-normalizer of $R$ in $\cal F$
(see \cite[2.14]{P01}).
By \cite[Corollary 3.6]{P01}, $N_{\cal F}^K(R)$ is exactly the Brauer category
${\cal F}_{\big(N_P^K(R),\, b_N^K\big)}\big(N_G^K(R),\, b_R\big)$.

Set $d=b_R$. We regard $d$ as a block of $N_G^K(R)$.
For any subgroup $T$ of $N_P^K(R)$,
denote by $d_T$ the block of $C_{N_G^K(R)}\big(N_P^K(R)\big)$ such that
$(T,\, d_T)$ is the Brauer pair contained in $\big(N_P^K(R),\, b_N^K\big)$.
By \cite[1.3]{P00}, the hyperfocal $Q'$ of the block $b_R$ of $N_G^K(R)$ with respect to the maximal Brauer pair $\big(N_P^K(R),\, b_N^K\big)$
is generated by all these commutators $[x,\, T]$,
where $T$ runs over subgroups of $N_P^K(R)$ and
$x$ runs over the set of $p'$-elements of $N_{N_G^K(R)}(T,\, d_T)$.
Given a $p'$-element $z$ of $N_{N_G^K(R)}(T,\, d_T)$,
by the definition of $N_{\cal F}^K(R)$,
there is $y\in N_G(TR, \, b_{TR})$ such that
$z$ and $y$ induce the same automorphism on $T$ by conjugation;
moreover, we may adjust the choice of $y$ so that $y$ is a $p'$-element.
Then we have $[z,\, T]=[y,\, T]\subset [y,\, RT]\subset Q$, and $Q'\subset Q$.
\end{proof}

\begin{lem}\label{inherition}
Keep the notation in the paragraph above
Lemma \ref{hyperfocal subgroup of local subgroup},
and assume that the block $b$ is controlled by $N_G(P)$.
Then, the Brauer category of the block $b_R$ of $N_G^K(R)$ is controlled by the normalizer of $N_P^K(R)$ in $N_G^K(R)$.
\end{lem}

\begin{proof} Let $(U,\, g)$ be a maximal Brauer pair of the block $b_R$ of $N_G^K(R)$. Since $C_G(RU)=C_{C_G(R)}(U)$ is normal in $C_{N_G^K(R)}(U)$,
there is a block $e$ of $C_G(RU)$ such that $eg\neq 0$.
Since ${\rm Br}_U(b_R){\rm Br}_U(g)={\rm Br}_U(g)$ and
${\rm Br}_U(b_R)\in k C_{C_G(R)}(U)\subset k C_{N_G^K(R)}(U)$,
where ${\rm Br}_U$ denotes the Brauer homomorphism
$({\cal O} G)^U\rightarrow kC_G(U)$ and
$({\cal O} G)^U$ is the subalgebra of all elements in ${\cal O} G$ commuting with $U$,
${\rm Br}_U(b_R){\rm Br}_U(e)={\rm Br}_U(e)$.
This implies that $(RU,\, e)$ is a $b$-Brauer pair and contains $(R, \, b_R)$.
There is some $x\in G$ such that $(RU,\, e)^x\leq (P,\,b_P)$.
So $(R,\, b_R)^x\leq (P,\,b_P)$.
Since the block $b$ is controlled by $N_G(P)$,
there are $y\in C_G(R)$ and $z\in N_G(P,\, b_P)$ such that $x=yz$.
So $(RU,\, e)^y\leq (P,\,b_P)$ and $P\cap N_G^K(R)=U^y=N_P^K(R)$ is a defect group of the block $b_R$ of $N_G^K(R)$.
In particular, $(R,\,b_R)$ is fully $K$-normalized in $\mathcal{F}_{(P,\,b_P)}(G,\,b)$.

Assume that $U=N_P^K(R)$. Set ${\cal F}=\mathcal{F}_{(P,\,b_P)}(G,\,b)$.
By \cite[Corollary 3.6]{P01}, we may choose a suitable $g$,
so that $N_{\cal F}^K(R)$ is the Brauer category
${\cal F}_{(U,\, g)}\big(N_G^K(R),\, b_R\big)$.
Set $d=b_R$. We regard $d$ as a block of $N_G^K(R)$.
For any subgroup $T$ of $N_P^K(R)$,
denote by $d_T$ the block of $C_{N_G^K(R)}\big(N_P^K(R)\big)$
such that $(T,\, d_T)$ is the Brauer pair contained in $(U,\, g)$.
For any $x\in N_G^K(R)$ such that $(T,\,d_T)^x=(T,\,d_T)$,
by the definition of $N_{\cal F}^K(R)$, there is $y\in N_G^K(R)\cap N_G(P,\, b_P)$
such that $y$ and $x$ induce the same automorphism on $T$.
Since $N_G^K(R)\cap N_G(P,\, b_P)\subset N_{N_G^K(R)}(U)$,
the lemma follows from Alperin's fusion theorem.
\end{proof}

\begin{lem}\label{inertial subgroup of local subgroups}
  Assume that the block $b$ is controlled by $N_G(P)$. The following hold.

\noindent{\bf (i)} The quotient group $N_G(P,\,b_P)/C_G(P)$ has a suitable
$p$-complement $E/C_G(P)$ such that
 $N_E(R)/C_G(P)$ is isomorphic to a $p$-complement of the quotient group
 $N_{N_G(R,\, b_R)}\big(N_P(R)\big)/C_{N_G(R,\, b_R)}\big(N_P(R)\big)$.

\noindent{\bf (ii)}
 The block $b_R$ of $RC_G(R)$ is not nilpotent if and only if
 the quotient group $N_G(P,\,b_P)/C_G(P)$ has a suitable $2$-complement $E/C_G(P)$ such that $E$ centralizes $R$ and the group $E/C_G(P)$ is isomorphic to a
 $2$-complement of
 $N_{RC_G(R)}\big(RC_P(R),\, b_{RC_P(R)}\big)/C_{RC_G(R)}\big(RC_P(R)\big)$.
\end{lem}

\begin{proof} {\bf (i)}
Set $E'=E/C_G(P)$ and $G'=P\rtimes E'$, the semidirect product of $P$ by $E'$.
Since the block $b$ is controlled by $N_G(P)$, the correspondence
$(T,\,b_T)\mapsto T$ induces an isomorphism between
the Brauer category $\mathcal{F}_{(P,\,b_P)}(G,\,b)$ and
the fusion system $\cal F$ of the group $G'$ with respect to the Sylow $p$-subgroup $P$. The $b$-Brauer pair $(R,\,b_R)$ is fully normalized in $\mathcal{F}_{(P,\,b_P)}(G,\,b)$ (see the first paragraph of the proof of Lemma \ref{inherition}).
By \cite[Corollary 3.6]{P01}, the Brauer category
${\cal F}_{\big(N_P(R), b_{N_P(R)}\big)}\big(N_G(R,\,b_R),\,b_R\big)$
is isomorphic to the fusion system of the group $N_{G'}(R)$ with respect to
the Sylow $p$-subgroup $N_P(R)$.
Clearly $N_{G'}(R)=N_P(R)\rtimes N_{E'}(R)$.
By \cite[CH. 5, Theorem 3.4]{G68},
the centralizer of $N_P(R)$ in $N_{E'}(R)$ is trivial.
So $N_{E'}(R)$ is isomorphic to a $p$-complement of
the automorphism group of $N_P(R)$ in the fusion system $\cal F$.
Note that $N_{E'}(R)=N_E(R)/C_G(P)$. The proof is done.

{\bf (ii)} By a proof similar to the proof of the statement {\bf (i)},
we prove that the quotient group $N_G(P,\,b_P)/C_G(P)$ has
a suitable $p$-complement $E/C_G(P)$ such that
 $C_E(R)/C_G(P)$ is isomorphic to a $p$-complement of
 the quotient group
 $N_{RC_G(R)}\big(RC_P(R),\,b_{RC_P(R)}\big)/C_G\big(R C_P(R)\big)$.
 Assume that the block $b_R$ of $RC_G(R)$ is not nilpotent.
Since the block $b_R$ of $RC_G(R)$ is controlled by the normalizer of
its defect group $RC_P(R)$ (see Lemma \ref{inherition}),
$C_E(R)$ is not equal to $C_G(P)$;
otherwise,
$N_{RC_G(R)}\big(RC_P(R),\,b_{RC_P(R)}\big)=C_G\big(R C_P(R)\big)$ and
the block $b_R$ of $RC_G(R)$ is nilpotent.
Since the quotient group $E/C_G(P)$ has order $3$,
we have $E=C_E(R)$ and thus $E$ centralizes $R$.
Now the necessity of the statement {\bf (ii)} is proved.
The sufficiency of the statement {\bf (ii)} is trivial.
\end{proof}

\begin{lem}\label{Nilpotent}
Assume that the quotient group $N_G(Q,\, b_Q)/C_G(Q)$ is a $p$-group.
Then the block $b$ is nilpotent.
\end{lem}

\begin{proof}
By the assumption, $N_G(Q,\, b_Q)$ is equal to $PC_{G}(Q)$.
Since $b_Q$ as a block of $C_G(Q)$ is nilpotent,
so is $b_Q$ as a block of $N_G(Q,\, b_Q)$.
Then by \cite[Theorem 2]{W14}, the block $b$ is nilpotent.
\end{proof}

Let $R$ be a normal $2$-subgroup of $G$ such that $|G:C_G(R)|$ is a $2$-power.
Set $\bar G=G/R$ and let $\bar{b}$ be the image of $b$ in ${\cal O} \bar{G}$.
Then $\bar b$ is a block of $\bar G$ with defect group $\bar P=P/R$.
Denote by $K$ the converse image of $C_{\bar G}(\bar P)$ in $G$.
We have $C_G(P)\leq K\leq N_G(P)$ and the index of $C_G(P)$ in $K$ is a $2$-power. There is a unique block $\tilde b_P$ of $K$ covering $b_P$.
Denote by $\bar b_{\bar P}$ the image of $\tilde b_P$ in ${\cal O} C_{\bar G}(\bar P)$. The pair $(\bar P, \, \bar b_{\bar P})$ is a maximal $\bar b$-Brauer pair.

\begin{lem}\label{reduced to the trivial subgroup}
Keep the notation
and assumption in the paragraph above.
Then $Q\cap R=1$ and $\bar{Q}=QR/R$ is the hyperfocal subgroup of $\bar{b}$
with respect to the maximal $\bar b$-Brauer pair $(\bar P, \, \bar b_{\bar P})$.
\end{lem}

\begin{proof}
We firstly prove the intersection $Q\cap R=1$.
Suppose that the block $b$ is not controlled by $N_G(P)$.
By Proposition \ref{fusion system},
there is a unique essential object $(S,\, b_S)$ in $\mathcal{F}_{(P,\,b_P)}(G,\,b)$.
We fix a subgroup $E_S$ of $N_G(S,\, b_S)$ as in the paragraph
above Lemma \ref{P=QC_P(E)}.
Clearly $N_G(S,\, b_S)\subset N_G(Q,\,b_Q)$ and by Lemma \ref{P=QC_P(E)},
the inclusion $E_S\subset N_G(Q,\,b_Q)$ induces an injective group homomorphism $E_S/C_G(S)\rightarrow N_G(Q,\, b_Q)/C_G(Q)$.
In particular, $E_S$ acts freely on the set of nontrivial elements of $Q$.
On the other hand, since $G/C_G(R)$ is a $2$-group,
$E_S$ has to be contained in $C_G(R)$.
Therefore the intersection of $Q$ and $R$ is trivial.
Suppose that the block $b$ is  controlled by $N_G(P)$.
Then we similarly prove the equality $Q\cap R=1$, replacing $(S,\, b_S)$ by $(P,\, b_P)$ and $E_S$ by $E$ in the paragraph under the proof of Proposition \ref{fusion system}. Summarizing the above, we now have the equality $Q\cap R=1$.

Let $\tilde Q$ be the hyperfocal subgroup of $\bar{b}$ with respect to
the maximal $\bar b$-Brauer pair $(\bar P, \, \bar b_{\bar P})$.
Using the second paragraph and the third paragraph in the proof of \cite[Lemma 8]{W14}, we prove that $\tilde Q$ is contained in $\bar{Q}$.
By Lemma \ref{P=QC_P(E)}, there always exists a subgroup $T$ of $P$ such that
$R\leq T$ and $Q=[T,\,O^2\big(N_G(T,\,b_T)\big)]$.
Since the image of $N_G(T,\,b_T)$ in $\bar G$ is contained in
$N_{\bar G}(\bar T,\, \bar b_{\bar T})$, we have $\bar{Q}\subset \tilde Q$.
Therefore $\bar{Q}= \tilde Q$.
\end{proof}

\section{Lower defect groups of the block $b$}

\bigskip\quad\, Lower defect groups of blocks associated with $p$-sections and their multiplicities are defined in \cite{F82} in the paragraph above \cite[Chapter V, Lemma 10.8]{F82}.
In order to avoid lengthening the paper, w
e omit an introduction to them.
Let $R$ be a lower defect group of $b$ associated with the identity element of $G$,
and denote by $m(b,\,R)$ its multiplicity.
By \cite[Chapter V, Corollary 10.13]{F82}, we have
\begin{equation}\label{formu}
l(b)=\sum_{R}m(b,\,R),
\end{equation}
where $R$ runs through a set of representatives for
the $G$-conjugacy classes of $p$-subgroups of $G$.
There is another formula on the number $l(b)$ in terms of
lower defect groups of blocks of local subgroups.
Let $\mathcal{T}$ be a set of subgroups of $P$ such that
$\{(T,\,b_T)\mid T\in\mathcal{T}\}$ is exactly a set of representatives for
the $G$-conjugacy classes of $b$-Brauer pairs.
By \cite[Section 4]{W14}, we have
\begin{equation}\label{formular of l(b)}
  l(b)=\sum_{T\in\mathcal{T}}m\big((b_T)^{N_G(T,\,b_T)},\,T\big).
\end{equation}
See the definition of $(b_T)^{N_G(T,\,b_T)}$ in the first paragraph in \cite[\S 14]{A86}.

\begin{lem}\label{ldg&qg}
Let $R$ be a normal $p$-subgroup of $G$ such that $|G:C_G(R)|$ is a $p$-power.
Set $\bar{G}=G/R$ and let $\bar{b}$ be the block of $\bar{G}$ determined by the block $b$ of $G$.
Then $m(b,\,R)=m(\bar{b},\,1)$.
\end{lem}

\begin{proof}
Let $\mathrm{Cl}(G_{p'})$ be the set of $p$-regular classes of $G$
and $\mathrm{Bl}(G)$ the set of blocks of $G$.
There is a block partition
$\bigcup_{b\in\mathrm{Bl}(G)}\Lambda_b$ of $\mathrm{Cl}(G_{p'})$.
Thus $m(b,\,R)=|\{C\in\Lambda_b\,|\,D(C)=_GR\}|$,
where $D(C)$ denotes a defect group of $C$.
By \cite[Theorem 5.11.6]{NT89},
the elementary divisors of the Cartan matrix $C_b$ of the block $b$
are given by $\{|D(C)|\,\big|\,C\in\Lambda_b\}$.
Since $|G:C_G(R)|$ is a $p$-power,
by \cite[Theorem 5.8.11]{NT89}
$C_b=|R|C_{\bar{b}}$,
where $C_{\bar{b}}$ denotes the Cartan matrix of the block $\bar{b}$.
Thus we have $m(b,\,R)=m(\bar{b},\,1)$
\end{proof}

\begin{lem}\label{P=Q lower defect group}
Assume that $P=Q=\mathbb Z_{2^n}\times \mathbb Z_{2^n}$.
We have
$m(b,\,R)=
\left\{
  \begin{array}{ll}
    1 & \mbox{\rm if $R$ is conjugate to $P$; } \\
    2 & \mbox{\rm if $R$ is equal to $1$; } \\
    0 & \mbox{\rm otherwise.}
  \end{array}
  \right.
$
\end{lem}

\begin{proof}
Clearly the block $b$ is not nilpotent.
By \cite{B71} and \cite{EKKS}, we have $l(b)=3$.
Since $m(b,\,P)=1$, in order to prove the lemma, by the equality (\ref{formu}),
it suffices to show $m(b, \,R)=0$ for any nontrivial proper subgroup $R$ of $P$.
In order to do that, by \cite[Theorem 7.2]{O80},
it suffices to show $m(c,\, R)=0$ for any block $c$ of $N_G(R)$ such that $c^G=b$.
Let $d$ be a block of $C_G(R)$ covered by $c$.
By the equality (10) in \cite{W14}, we have $m(c,\, R)=m\big(d^{N_G(R,\,d)},\, R\big)$. The pair $(d,\, R)$ is a $b$-Brauer pair and is conjugate to
a Brauer pair contained in $(P,\, b_P)$.
So in order to prove $m(c,\, R)=0$,
it suffices to show $m\big(b_R^{N_G(R,\, b_R)},\, R\big)=0$.
Suppose that $m\big(b_R^{N_G(R,\, b_R)},\, R\big)\geq1$.
By \cite[Theorem 5.12]{O80},
we have $m(b_R,\, R)\geq m\big(b_R^{N_G(R,\,b_R)},\,R\big)\geq 1$.
However, since $P=\mathbb Z_{2^n}\times \mathbb Z_{2^n}$,
the block $b_R$ of $C_G(R)$ is nilpotent,
it has only one simple module up to isomorphism, and thus $m(b_R,\, R)$ has to be 0. That causes a contradiction.
\end{proof}

We borrow the subgroup $E$ of $N_G(P, \,b_P)$ in the paragraph under
Proposition \ref{fusion system},and set $R=C_P(E)$.

\begin{lem}\label{lower defect group C_P(E)}
Keep the notation as above and assume that the block $b$ is controlled by $N_G(P)$. Then we have $m\big(b_R^{N_G(R,\, b_R)},\, R\big)=2$.
\end{lem}

\begin{proof}
Set $G'=N_G(R, \,b_R)$, $b'=(b_R)^{G'}$, $P'=N_P(R)$ and $b'_{P'}=b_{N_P(R)}$.
By the first paragraph of the proof of Lemma \ref{inherition},
$(P',\, b'_{P'})$ is a maximal $b'$-Brauer pair.
Obviously $E$ is contained in $N_{G'}(P', \,b'_{P'})$ and
by the proof of Lemma \ref{inertial subgroup of local subgroups} {\bf (i)},
the inclusion induces an isomorphism from $E/C_G(P)$ to a $2$-complement of
$N_{G'}(P',\, b'_{P'})/C_{G'}(P')$.
On the other hand, by Lemma \ref{inherition},
the block $b'$ is controlled by the normalizer $N_{G'}(P')$.
Therefore the hyperfocal subgroup $Q'$ of the block $b'$ with respect to
$(P',\, b'_{P'})$ is not trivial.
By Lemma \ref{hyperfocal subgroup of local subgroup}, $Q'$ is contained in $Q$ and
thus it is equal to $\mathbb Z_{2^m}\times \mathbb Z_{2^l}$ for integers $m$ and $l$.
We claim $m=l$. Otherwise, by Lemma \ref{Nilpotent},
the block $b'$ is nilpotent and thus $Q'$ is equal to 1.
That is against $Q'$ being nontrivial.

Set $E'=E  C_{G'}(P')$.
The quotient group $E'/C_{G'}(P')$ is a $2$-complement of
$N_{G'}(P',\,b'_{P'})/C_{G'}(P')$.
Clearly $C_{P'}(E')$ is equal to $R$.
Since the block $b'$ is controlled by $N_{G'}(P')$, by Lemma \ref{P=QC_P(E)},
we have $Q'=[P', \,E']$ and $P'=Q'\rtimes R$.
So $Q'$ normalizes $R$ and it is contained in $C_Q(R)$.
Since $P=Q\rtimes R$, we have $P'=C_Q(R)\times R$.
Then the equalities $P'=C_Q(R)\times R=Q'\rtimes R$ and
the inclusion $Q'\subset C_Q(R)$ force $Q'=C_Q(R)$.

Since the block $b$ is controlled by $N_G(P)$,
we have $G'=C_G(R)\big(N_G(P,\,b_P)\cap G'\big)$.
Since $N_G(P,\,b_P)=PE$ and $E$ is contained in $G'$,
we have $G'=P'C_G(R)=RC_G(R)$.
Set $\bar G'=G'/R$ and let $\bar b'$ be the image of $b'$ in ${\cal O} \bar G'$.
Then $\bar b'$ is a block of $G'$ and by Lemma \ref{reduced to the trivial subgroup}, $Q'R/R$ is a hyperfocal subgroup of the block $\bar b'$, which is equal to $P'/R$.
By Lemmas \ref{ldg&qg} and \ref{P=Q lower defect group},
we have $m(b',\,R)=m(\bar{b'},\,1)=2$.
\end{proof}

\begin{lem}\label{wreathed product}
Assume that the block $b$ is not controlled by $N_G(P)$,
that $P$ is $\mathbb Z_{2^2}\wr \mathbb Z_2$ and
that $Q$ is $\mathbb Z_{2^2}\times \mathbb Z_{2^2}$.
Then we have $m(b,\,1)=1$.
\end{lem}

\begin{proof}
The center $Z(P)$ is cyclic.
Let $X$ be the unique subgroup of $Z(P)$ of order $2$.
Clearly $(P, \,b_P)$ is a maximal Brauer pair of the block $b_X$ of $C_G(X)$.
Denote by $Q_X$ the hyperfocal subgroup of the block $b_X$ of $C_G(X)$
with respect to the maximal $b_X$-Brauer pair $(P, \,b_P)$.
Obviously $(Q_X,\, b_{Q_X})$ is a $b_X$-Brauer pair.
We claim that the quotient group
$N_{C_G(X)}(Q_X,\,b_{Q_X})/C_G(Q_X)$ is a $2$-group.
By Lemma \ref{hyperfocal subgroup of local subgroup},
$Q_X$ is contained in $Q$,
so $Q_X$ is isomorphic to $\mathbb Z_{2^m}\times \mathbb Z_{2^l}$ for
some integers $m$ and $l$ such that $0\leq m, l \leq 2$.
Suppose that $m\neq l$ or $m=l=0$.
Then the claim is clear.
Suppose that $1\leq m= l \leq 2$.
Then $Q_X$ is $Q$ or $Q_0$.
By Lemma \ref{Aut(Q)}, the quotient group $N_{C_G(X)}(Q_X,\,b_{Q_X})/C_G(Q_X)$ has no $3$-element since any $3$-element of ${\rm Aut}(Q_X)$ acts transitively on
the three nontrivial elements of $Q_0$
while any element in $N_{C_G(X)}(Q_X,\,b_{Q_X})/C_G(Q_X)$ centralizes $X$.
So in this case $N_{C_G(X)}(Q_X,\,b_{Q_X})/C_G(Q_X)$ is a $2$-group too.
The claim is done.
Then by Lemma \ref{Nilpotent}, the block $b_X$ is nilpotent.
By \cite[Lemma (5.A)]{K80},
an elementary divisor of the Cartan matrix of the block $b$ is either $1$ or $|P|$.

Clearly $P$ is generated by elements $x,y,z$ with the relations:
$|x|=|y|=4,\, |z|=2,\, xy=yx,\, zxz=y$ and $Q$ is generated by $x,y$.
Let $Y$ be the subgroup of $P$ generated by elements $xy,x^2,z$.
Set $Q_1=\langle xy,x^2\rangle$, $Q_2=\langle xy,x^2z\rangle$ and
$Q_3=\langle xy,z\rangle$.
Since the block $b$ is controlled by $N_G(Q,\,b_Q)$,
for any $w\in N_G(Y,\,b_Y)$, ${Q_1}^w$ is contained in $Q$.
This shows that $Q_1$, $Q_2$ and $Q_3$ are not $N_G(Y,\,b_Y)$-conjugate.
By \cite[Lemma (3.D)]{K80}, the order of $N_G(Y,\,b_Y)/YC_G(Y)$ is $2$.
Since the block $b$ is not controlled by $N_G(P)$,
the order of $N_G(Q,\,b_Q)/C_G(Q)$ has to be $6$.
By \cite[Proposition 14.E]{K80}, we have $l(b)=2$.
Now by the equality (\ref{formu}), we have $m(b,\,1)=1$.
\end{proof}

Assume that the block $b$ is not controlled by $N_G(P)$.
By Proposition \ref{fusion system}, there exists a unique essential object $(S, b_S)$ in
$\mathcal{F}_{(G,\,b)}(P,\,b_P)$;
moreover, borrowing the subgroup $E_S$ of $N_G(S,\,b_S)$ in the paragraph above Lemma \ref{P=QC_P(E)}, we have $S=Q\rtimes C_S(E_S)$.
Set $T=C_S(E_S)$.

\begin{lem}\label{lower defect group C_S(E_S)}
Keep the notation and the assumption as above. Assume that $|Q|$ is less than $16$.
Then we have $m\big(b_T^{N_G(T,\,b_T)}, T\big)=1$.
\end{lem}

\begin{proof}
Set $G'=N_G(T,\,b_T)$ and $b'=(b_T)^{G'}$.
Since the block $b$ is controlled by $N_G(S,\,b_S)$,
there exists $x\in N_G(S,\,b_S)$
such that $N_{P^x}(T)$ is a defect group of the block $b'$ of $G'$.
Since $T=C_S(E_S)$, we may adjust the choice of $E_S$
by the $N_G(S,\,b_S)$-conjugation,
so that $N_P(T)$ is a defect group of the block $b'$.
Then the pair $\big(N_P(T),\,b_{N_P(T)}\big)$ is a maximal $b'$-Brauer pair.
Set $P'=N_P(T)$ and $b'_{P'}=b_{N_P(T)}$.
Denote by $Q'$ the hyperfocal subgroup of the block $b'$ with respect to
the maximal $b'$-Brauer pair $(P',\,b'_{P'})$.
By Lemma \ref{hyperfocal subgroup of local subgroup}, $Q'$ is contained in $Q$.

The subgroup $E_S$ acts transitively on $Q_0-\{1\}$, $E_S\cap C_G(Q_0)=C_G(S)$, and the $G$-conjugation induces an isomorphism from $E_SC_G(Q_0)/C_G(Q_0)$ to the Sylow $3$-subgroup of
$\mathrm{Aut}(Q_0)$ (see the second paragraph above Lemma \ref{P=QC_P(E)}).
By Proposition \ref{fusion system},
we have $S=C_P(Q_0)$.
Thus $Q_0$ is contained in $P'$.
Clearly, we have $(TQ_0,\,b'_{TQ_0})=(TQ_0,\,b_{TQ_0})$.
Since $E_S$ is contained in $ N_G(S,\,b_S)$
and the block $b$ is controlled by $N_G(S,\,b_S)$,
$E_S\leq N_{G'}(TQ_0,\,b'_{TQ_0})\leq N_{G'}(Q_0,\,b'_{Q_0})$.
So we have $Q_0=[Q_0,\,E_S]\leq Q'$.
Now we claim that $Q'$ is $\mathbb Z_{2^m}\times \mathbb Z_{2^m}$ for
some $1\leq m\leq 2$.
Otherwise, by Lemma \ref{Nilpotent}, the block $b'$ is nilpotent and
thus $Q'$ has to be 1.
This contradicts with the inclusion $Q_0\subset Q'$.

Recall that $N_G(S,\,b_S)$ has a subgroup $A$ containing $SC_G(S)$
such that $A/SC_G(S)=O_3\big(N_G(S,\,b_S)/SC_G(S)\big)$
and
$E_S/C_G(S)$ is a $2$-complement of $A/C_G(S)$.
Clearly $P$ is contained in $N_G(A)$ and
for any $u\in P-S$, $(E_S)^u/C_G(S)$ is a $2$-complement of $A/C_G(S)$ too.
By the Schur-Zassenhaus theorem,
there exists $v\in S$ such that $(E_S)^u/C_G(S)=(E_S)^v/C_G(S)$.
This implies that $uv^{-1}\in P'$ and that $P'$ is not contained in $S$.
Since $S=C_P(Q_0)$, $Q_0$ is not contained in $Z(P')$.
By Proposition \ref{fusion system},
$\mathcal{F}_{(P',\, b'_{P'})}(G',\,b')$ has a unique essential object $(S',\,b'_{S'})$. Obviously $N_S(T)$ is a subgroup of $P'$ centralizing $Q_0$ and
since $Q_0$ is not contained in $Z(P')$, $N_S(T)$ is a proper subgroup of $P'$.
Since $|P:S|=2$, $|P':N_S(T)|$ is less than $2$ and
thus $N_S(T)$ is forced to be equal to $S'$.
By Lemma \ref{P=QC_P(E)}, we have $S'=C_Q(T)\times T$.

Clearly $(S',\,b'_{S'})=(S',\,b_{S'})$, $E_S\subset N_{G'}(S', b'_{S'})$ and
$C_G(S)\subset C_G(S')\subset N_{G'}(S',\,b'_{S'})$.
Since the group $E_S/C_G(S)$ is of order $3$ and $C_G(S)\subset E_S\cap C_G(S')$, the group $E_SC_G(S')/C_G(S')$ has to be of order $1$ or $3$.
If $E_SC_G(S')/C_G(S')$ is of order $1$,
then we have $E_S\subset C_G(S')\subset C_{G'}(Q')$.
But this is impossible.
So the group $E_SC_G(S')/C_G(S')$ has to be of order $3$ and
the inclusion $E_S\subset N_{G'}(S',\,b'_{S'})$ induces an injective group homomorphism $E_SC_G(S')/C_G(S')\rightarrow N_{G'}(S',\,b'_{S'})/S'C_{G'}(S')$.
Since $N_{G'}(S',\,b'_{S'})/S'C_{G'}(S')$ is isomorphic to ${\rm S}_3$,
$E_SC_G(S')/C_G(S')$ is a $2$-complement of the converse image of
$O_3\big(N_{G'}(S',\,b'_{S'})/S'C_{G'}(S')\big)$ in $N_{G'}(S',\,b'_{S'})/C_{G'}(S')$.
Take $E_{S'}=E_SC_G(S')$. We have $C_{S'}(E_{S'})=T$.
Since $Q'\leq Q$, $Q'$ is contained in $C_Q(T)$.
By Lemma \ref{P=QC_P(E)}, $Q'$ has to be equal to $C_Q(T)$.

Since the block $b$ is controlled by $N_G(S)$, we have
$G'=C_G(T)\big(N_G(S,\,b_S)\cap G'\big)=C_G(T)\big(PE_S\cap G'\big)=P'C_G(T)$.
Let $\bar b'$ be the image of $b'$ through the obvious surjective homomorphism
${\cal O} G'\rightarrow {\cal O}  G'/T$.
Then $\bar b'$ is a block of $G'/T$ and
we have $m(b', \,T)=m(\bar{b'},\, 1)$ by Lemma \ref{ldg&qg}.
By Lemma \ref{reduced to the trivial subgroup},
$P'/T$ is a defect group of $\bar b'$
and $Q'T/T=S'/T$ is a hyperfocal subgroup of $\bar b'$.
Set $\bar P'=P'/T$ and $\bar Q'=Q'T/T$.
Then $\bar Q'$ is $\mathbb Z_{2^m}\times \mathbb Z_{2^m}$
and the index of $\bar Q'$ in $\bar P'$ is $2$ since the index of $S'$ in $P'$ is $2$.
By \cite[Proposition 2.7]{S16}, we have
either $\bar Q'=\mathbb Z_2\times \mathbb Z_2$ and $\bar P'=D_8$
or $\bar Q'=\mathbb Z_{2^2}\times \mathbb Z_{2^2}$ and
$\bar P'= \mathbb Z_{2^2}\wr \mathbb Z_2$.
In the former case, we exclude the case (aa) of \cite[Lemma 9.3]{S14}
since in that case the hyperfocal subgroup of the corresponding fusion system over
$\bar P'$ is $\bar P'$ itself;
then by \cite[Theorem 9.7]{S14}, $l(\bar b')$ is $2$;
by \cite[Theorem 8.1 (2)]{S14},
the Cartan matrix $C'$ of the block $\bar b'$ is
$$M^{\mathrm{T}}
\left(
 \begin{array}{cc}
            3 & 2 \\
            2 & 4 \\
  \end{array}
    \right)M$$
for some
$M\in\mathrm{GL}(2,\,\mathbb{Z})$,
which means the elementary divisors of $C'$ are $1$ and $8$;
since $m(\bar b',\, \bar P')=1$, by the equality (\ref{formu}),                                                                                                                          we have $m(\bar b',\,1)=1$.
In the latter case, we use Lemma \ref{wreathed product} to get $m(\bar b',\,1)=1$.
\end{proof}

\begin{lem}\label{vertex}
Assume that there is a simple ${\cal O}Gb$-module $M$ with a vertex $Q$.
Then we have
$$Q=\left\{
  \begin{array}{ll}
    P & \mbox{\rm if the block $b$ is controlled by $N_G(P)$ and $Q$ is contained in $Z(P)$; } \\
    S & \mbox{\rm if the block $b$ is not controlled by $N_G(P)$ and $Q$ is contained in $Z(S)$. } \\
    \end{array}\right.
$$
\end{lem}

\begin{proof}
Since the simple ${\cal O}G b$-module $M$ has vertex $Q$,
by \cite[Corollary 41.7]{T95} there is a selfcentralizing $b$-Brauer pair $(Q,\,e)$.
Assume that the block $b$ is controlled by $N_G(P)$ and $Q$ is contained in $Z(P)$. Then $P$ is contained in $C_G(Q)$ and there is a maximal $b$-Brauer pair $(P, \,g)$ such that $(Q,\,e)$ is contained in $(P,\, g)$.
By \cite[Proposition 41.3]{T95}, we have $P=C_P(Q)=Q$.
If the block $b$ is not controlled by $N_G(P)$ and $Q$ is contained in $Z(S)$,
similarly we have $S=Q$.
\end{proof}

\begin{prop}\label{m(b,1)}
Assume that $Q$ is strictly contained in $P$.
Furthermore, assume that $Q$ is contained in $Z(P)$ if
the block $b$ is controlled by $N_G(P)$,
and that $Q$ is contained in $Z(S)$ if the block $b$ is not controlled by $N_G(P)$.
Then either any Cartan integer of $b$ is divisible by $2$ and then $m(b,\,1)$ vanishes,
or $Q$ is equal to $S$.
\end{prop}

\begin{proof}
Assume that no simple ${\cal O} Gb$-module is relatively projective to $Q$.
By the proof of \cite[Theorem 4]{W14}, any Cartan integer of $b$ is divisible by $2$ and $m(b,\,1)$ vanishes.
Assume that ${\cal O} Gb$ has a simple module relatively projective to $Q$.
By \cite[Corollary 41.7]{T95}, $Q$ has to be a vertex of the simple ${\cal O}Gb$-module. Since $Q$ is strictly contained in $P$, by Lemma \ref{vertex}, $Q$ has to be $S$.
\end{proof}

\section{Proof of Theorem \ref{MT}}

\bigskip\quad\, In this section, we borrow the subgroups $E$ and $E_S$ in the paragraphs above Lemma \ref{P=QC_P(E)}.

\medskip\begin{lem}\label{ctrlb}
Assume that the block $b$ is controlled by $N_G(P)$ and $Q$ is contained in $Z(P)$.
We have $l(b)=3$.
\end{lem}

\begin{proof}
We prove this by induction on $|G|$.
The lemma is known by \cite{B71} and \cite{EKKS} when $P=Q$.
So we assume that $Q$ is strictly contained in $P$.
Let $R$ be a proper subgroup of $P$ such that
 $m\big((b_R)^{N_G(R,\,b_R)}, R\big)\neq0$
Set $N=N_G(R,\,b_R)$, $d=(b_R)^N$ and $\hat{R}=N_P(R)$.
By Lemma \ref{inherition} and its proof,
the pair $(\hat{R},\,d_{\hat{R}})=(\hat{R},\,b_{\hat{R}})$ is a maximal $d$-Brauer pair and the block $d$ is controlled by $N_N(\hat{R},\,d_{\hat{R}})$.
Since $m(d,\,R)\neq 0$ and $R<P$, the block $d$ is not nilpotent.
So the hyperfocal subgroup $Q_d$ of the block $d$
with respect to  $(\hat{R},\,d_{\hat{R}})$ is not trivial.
By Lemma \ref{hyperfocal subgroup of local subgroup}, $Q_d$ is contained in $Q$; moreover, $Q_d$ has to be isomorphic to $\mathbb Z_{2^m}\times \mathbb Z_{2^m}$ for some $1\leq m\leq n$.
Otherwise, by Lemma \ref{Nilpotent}
the block $d$ is nilpotent which causes a contradiction.

Assume that $N<G$.
By induction we have $l(d)=3$.
By Lemma \ref{inertial subgroup of local subgroups},
the quotient group $N_G(P, \,b_P)/C_G(P)$ has a suitable $2$-complement
$E/C_G(P)$ such that $N_E(R)/C_G(P)$ is isomorphic to a $2$-complement of the quotient group $N_N(\hat R,\,d_{\hat R})/C_N(\hat R)$.
By Lemma \ref{lower defect group C_P(E)}
and  the equality (\ref{formular of l(b)}),
a proper subgroup $T$ of $\hat R$ such that $m(d,\,T)\neq 0$ has to be conjugate to $C_{\hat R}\big(N_E(R)\big)$.
So $R$ is equal to $C_{\hat R}(N_E(R))$.
Clearly both the quotient groups $N_E(R)/C_G(P)$ and $E/C_G(P)$ have order $3$ and $E$ is equal to $N_E(R)$.
Now we have $R=\hat{R}\cap C_P(E)=N_{C_P(E)}(R)$.
That forces $R=C_P(E)$.

Assume that $N=G$ and then $b=d$.
Since the block $b$ is controlled by $N_G(P,\,b_P)$, we have
$$G=C_G(R)\big(N_G(P,\,b_P)\cap N_G(R)\big)=C_G(R)N_G(P,\,b_P).$$

Assume that the block $b_R$ is nilpotent.
Then the block $(b_R)^{PC_G(R)}$ is also nilpotent,
$N_G(P,\,b_P)\cap PC_G(R)=PC_G(P)$,
and the quotient group $G/PC_G(R)$ is a $3$-group.
By the main theorem of \cite{Z16'}, the block $b$ is inertial (see \cite{Z16'}).
In particular we have $l(b)=3$.
Then by Lemma \ref{lower defect group C_P(E)}
and the equality (\ref{formular of l(b)}),
$R$ has to be equal to $C_P(E)$ for a suitable $2$-complement $E/C_G(P)$ of
the quotient group $N_G(P,\,b_P)/C_G(P)$.

Assume that the block $b_R$ is not nilpotent.
By Lemma \ref{inertial subgroup of local subgroups},
$R\leq C_P(E)$ for some $2$-complement $E/C_G(P)$ of the quotient group $N_G(P,\,b_P)/C_G(P)$.
So $G$ is equal $PC_G(R)$.
Set $\bar G=G/R$.
Let $\bar{b}$ be the image of $b$ under the surjective homomorphism
${\cal O} G\rightarrow {\cal O} \bar G$.
Then $\bar b$ is a block of $\bar G$,
$\bar{P}=P/R$ is a defect group of the block $\bar{b}$
and $(QR)/R\cong Q$ is a hyperfocal subgroup of $\bar{b}$
(see Lemma \ref{reduced to the trivial subgroup}).
By Lemma \ref{ldg&qg}, we have $m(\bar{b},\,1)=m(b,\,R)\neq0$.
Since $Q$ is contained in $Z(P)$,
by Proposition \ref{fusion system} {\bf{(a)}}
the block $\bar{b}$ is controlled by $N_{\bar{G}}(\bar{P})$.
By Proposition \ref{m(b,1)} applied to $\bar{G}$ and $\bar{b}$,
$(QR)/R$ has to be equal to $P/R$.
Now by Lemma \ref{P=QC_P(E)}, we have
$R=C_P(E)$.

In conclusion, we always have $R=C_P(E)$
for a suitable $2$-complement $E/C_G(P)$ of
the quotient group $N_G(P, \,b_P)/C_G(P)$.
Now the lemma follows from Lemma \ref{lower defect group C_P(E)}
and the equality (\ref{formular of l(b)}).
\end{proof}

\begin{lem}\label{Tctrlb}
Assume that the block $b$ is not controlled by $N_G(P)$,
that $Q$ is contained in the center of $S$,
that the $b$-Brauer pair $(T,\,b_T)$ is fully normalized in
$\mathcal{F}_{(P,\,b_P)}(G,\,b)$,
that $m\big((b_T)^{N_G(T,\,b_T)},T\big)\neq 0$ and
that $T$ is strictly contained in $P$.
Then the block $(b_T)^{N_G(T,\,b_T)}$ is not controlled by
$N_{N_G(T,\,b_T)}\big(N_P(T)\big)$.
\end{lem}

\begin{proof}
Set $N=N_G(T,\,b_T)$, $d=(b_T)^N$ and $\hat{T}=N_P(T)$.
Since $(T,\,b_T)$ is fully normalized in $\mathcal{F}_{(P,\,b_P)}(G,\,b)$,
the pair $(\hat{T},\,d_{\hat{T}})=(\hat{T},\,b_{\hat{T}})$ is a maximal $d$-Brauer pair. Assume that $\mathcal{F}_{(\hat{T},\, d_{\hat{T}})}(N,\,d)=
\mathcal{F}_{(\hat{T},\, d_{\hat{T}})}
\big(N_N(\hat{T},\,d_{\hat{T}}),\,d_{\hat{T}}\big).$

By Lemma \ref{hyperfocal subgroup of local subgroup},
the hyperfocal subgroup $Q_d$ of the block $d$ with respect to
the maximal $d$-Brauer pair $(\hat{T},\,d_{\hat{T}})$ is contained in $Q$.
Since $\mathcal{F}_{(\hat{T},\,d_{\hat{T}})}(N,\,d)=
\mathcal{F}_{(\hat{T},\,d_{\hat{T}})}\big(N_N(Q_d,\,d_{Q_d}),\,d_{Q_d}\big)$,
$Q_d$ is isomorphic to
$\mathbb Z_{2^m}\times \mathbb Z_{2^m}$ for some $1\leq m\leq n$;
otherwise, $N_N(Q_d,\,d_{Q_d})=\hat TC_N(Q_d)$ and the block $d$ is nilpotent; since $T$ is strictly contained in $P$, $m\big(d,\,T\big)$ has to be zero.

By Proposition \ref{fusion system},
we have $Q_0\leq Z(\hat{T})$, thus $T\leq \hat{T}\leq S$.
Since we are assuming that $Q$ is contained in the center of $S$,
$Q_d$ is contained in the center of $Z(\hat T)$.
By Lemma \ref{ctrlb}, we have $l(d)=3$.
Since $m(d,\,T)\neq 0$ and $T\neq \hat T$,
by Lemmas \ref{P=QC_P(E)} and \ref{lower defect group C_P(E)}
we have the equality $\hat{T}=Q_d\times T$, which implies $T\cap Q=1$.

We claim that there is a suitable $E_S$ such that $T=C_S(E_S)$.
Set $K=N_G(S,\,b_S)\cap N_N(\hat{T},\,b_{\hat{T}})$.
Since $T\leq \hat{T}\leq S$, $C_G(S)$ is normal in $K$.
Suppose that $K/C_G(S)$ is a $2$-group.
Since $PC_G(S)/C_G(S)$ is a Sylow  $2$-subgroup of $N_G(S,\,b_S)/C_G(S)$, there exists $x\in N_G(S,\,b_S)$ such that $K$ is contained in $P^xC_G(S)$.
By Proposition \ref{fusion system}, the block $b$ is controlled by $N_G(S,\,b_S)$.
So we have
$$N_N(\hat{T},\,b_{\hat{T}})=
N_G(T,\,b_T)\cap N_G(\hat{T},\,b_{\hat{T}})=
C_G(\hat{T})K=\hat{T}C_G(\hat{T}),$$
the block $d$ must be nilpotent and $m(d,\,T)=0$.
That contradicts with the inequality $m(d,\,T)\neq 0$.
Therefore, the quotient group $K/C_G(S)$ is not a $2$-group.

Take an element $y$ of $K$ such that
the coset $yC_G(S)$ in $K/C_G(S)$ has order $3$.
By the third paragraph above Lemma \ref{P=QC_P(E)},
we have $y\in E_S$ for some $E_S$.
Notice that $S=Q\rtimes C_{S}({E_S})$.
Suppose that there exists $t=ut^\prime\in T$ for some $u\in Q-\{1\}$ and
some $t^\prime\in C_{S}({E_S})$.
Then $[t,\,y]=[ut^\prime,\,y]=[u,\,y]$ is not equal to $1$ and
it is contained in $T\cap Q$.
That contradicts with the equality $T\cap Q=1$.
Hence $T\leq C_S({E_S}).$
Since $\hat T=Q_d\times T$ and $\hat T\leq S$,
we have $N_{C_S({E_S})}(T)=N_{C_S({E_S})}(T)\cap\hat{T}=T$ and
then $T=C_S({E_S})$.
The claim is done.

Using the third paragraph of the proof of
Lemma \ref{lower defect group C_S(E_S)},
we can prove $\hat{T}$ is not contained in $S$.
That contradicts with the inclusion $\hat T\subset S$.
\end{proof}

\begin{lem}\label{nctrlb}
Assume that the block $b$ is not controlled by $N_G(P)$,
that $Q$ is contained in $Z(S)$ and that $|Q|$ is less than $16$.
We have $l(b)=2$.
\end{lem}

\begin{proof}
We prove the lemma by induction on $|G|$.
Let $\mathcal{T}$ be a set of representatives for the $N_G(S,\,b_S)$-conjugacy classes of subgroups of $P$.
By Proposition \ref{fusion system},
$\{(T,\,b_T)\mid T\in\mathcal{T}\}$ is a set of representatives for the $G$-conjugacy classes of $b$-Brauer pairs.
For any $T\in\mathcal{T}$,
we further assume that $(T,\,b_T)$ is fully normalized in the Brauer category
${\cal F}_{(P,\,b_P)}(G,\,b)$.
Let $T<P$ and $T\in\mathcal{T}$ such that
$m\big((b_T)^{N_G(T,\,b_T)},\,T\big)\neq 0$.
Set $N=N_G(T,\,b_T)$, $d=(b_T)^N$ and $\hat{T}=N_P(T)$.
The pair $(\hat{T},\,d_{\hat{T}})=(\hat{T},\,b_{\hat{T}})$ is a maximal $d$-Brauer pair.
By Lemma \ref{hyperfocal subgroup of local subgroup},
the hyperfocal subgroup $Q_d$ of $d$ with respect to
the maximal $d$-Brauer pair $(\hat{T},\,d_{\hat{T}})$ is contained in $Q$.
Since $\mathcal{F}_{(\hat{T},\,d_{\hat{T}})}(N,\,d)=
\mathcal{F}_{(\hat{T},\,d_{\hat{T}})}\big(N_N(Q_d,\,d_{Q_d}),\,d_{Q_d}\big)$,
$Q_d$ is isomorphic to
$\mathbb Z_{2^m}\times \mathbb Z_{2^m}$ for some $1\leq m\leq n$;
otherwise, $N_N(Q_d,\,d_{Q_d})=TC_N(Q_d)$,
the block $d$ is nilpotent and $m(d,\,T)$ has to be zero.

Assume that $N<G$.
By Lemma \ref{Tctrlb} the block $d$ of $N$ is not controlled by $N_N(\hat T)$.
By induction, we have $l(d)=2$.
Let $(S_d,\,d_{S_d})$ be the unique essential object in
${\cal F}_{(\hat T,\,d_{\hat T})}(N,\,d)$.
By Proposition \ref{fusion system}, $Q_0\leq Z(S_d)$ and $S_d\leq S$.
On the other hand,
by Lemma \ref{lower defect group C_S(E_S)} $S_d=Q_d\times T$,
thus $T\cap Q=1$.
Now we use the fourth and fifth paragraphs of
the proof of Lemma \ref{Tctrlb} and prove $T=C_S(E_S)$ for a suitable $E_S$.

Assume that $N=G$.
Then $T$ is contained in $S$.
More precisely,  $T$ is contained in $C_S(E_S)$.
Otherwise,  $1\neq[T,\,E_S]\leq T\cap Q$.
Let $Q_{b_T}$ be a hyperfocal subgroup of $(b_T)^{TC_G(T)}$ contained in $Q$.
Assume that $Q_{b_T}\neq 1$.
Then the Sylow $3$-subgroup of $N_{TC_G(T)}(Q_{b_T})/C_{TC_G(T)}(Q_{b_T})$
is nontrivial.
On the other hand, since $T\cap Q\neq1$,
we have $1\neq T\cap Q_0\subset Q_{b_T}$.
Clearly $T\cap Q_0$ is in the center of $TC_G(T)$.
That forces the Sylow $3$-subgroup of
$N_{TC_G(T)}(Q_{b_T})/C_{TC_G(T)}(Q_{b_T})$ to act trivially on $T\cap Q_0$.
That is impossible.
So the block $(b_T)^{TC_G(T)}$ is nilpotent.
Since $0\neq m(b,\,T)\leq m\big((b_T)^{TC_G(T)},\,T\big)$
(see \cite[Theorem 5.12]{O80}), $C_P(T)$ has to be contained in $T$.
Then $N_G(S,\,b_S)\cap SC_G(T)=SC_G(S)$.
Since the block $b$ is controlled by $N_G(S,\,b_S)$,
we have $G=C_G(T)N_G(S,\,b_S)$.
So
$G/SC_G(T)\cong N_G(S,\,b_S)/\big(N_G(S,\,b_S)\cap SC_G(T)\big)\cong N_G(S,\,b_S)/SC_G(S)\cong \mathrm{S}_3$.
Denote by $H$ the normal subgroup of $G$ such that
$H$ contains $SC_G(T)$ and $|H:SC_G(T)|=3$.
Since the block $(b_T)^{SC_G(T)}$ is nilpotent,
$l\big((b_T)^H\big)=3$
by the structure of extensions of nilpotent blocks (see \cite{KP}).
So we have $l(b)=2$ or $3$.
If $l(b)=2$,
by Lemma \ref{lower defect group C_S(E_S)}
we have $T=C_S(E_S)$.
This is a contradiction.
Suppose that $l(b)=3$.
Denote by $C$ and $C_H$ the Cartan matrices of
the blocks $b$ and $(b_T)^H$ respectively.
Then it is easy to check that $C=2C_H$.
This implies that $m(b,\,1)\neq0$.
By the equality (\ref{formular of l(b)})
and Lemma \ref{lower defect group C_S(E_S)},
we still have $T=C_S(E_S)$.
This causes a contradiction.

Now we have $T\leq C_S(E_S)$ and $G=PC_G(T)$.
Let $\bar{b}$ be the unique block of $\bar{G}=G/T$ corresponding to $b$.
Then $\bar{P}=P/T$ is a defect group of $\bar{b}$ and
$\bar{Q}=QT/T\cong Q$ is a hyperfocal subgroup of $\bar{b}$.
At the same time, by Lemma \ref{ldg&qg},
we have $m(\bar{b},\,1)=m(b,\,T)\neq 0$.
Since $Q$ is strictly contained in $P$, by Proposition \ref{m(b,1)}
we have $|P:QT|=|\bar{P}:\bar{Q}|=2$.
So $QT=S$ and $T=C_S(E_S)$.

In conclusion, we always obtain that
$T=C_S(E_S)$ for a suitable $E_S$.
Now the lemma follows from Lemma \ref{lower defect group C_S(E_S)} and
the equality (\ref{formular of l(b)}).
\end{proof}

\begin{lem}\label{Ordinaryctrl}
Assume that the block $b$ is controlled by $N_G(P)$ and
that $Q$ is contained in the center of $P$.
We have $k(b)=k(c)=k(b_0)$.
\end{lem}

\begin{proof}
For any $u\in P$, set $d=b_{\langle u\rangle}$, $P_u=C_P(u)$ and $C=C_G(u)$,
and let $e_u$ be the inertial index of the block $d$.
By Lemma \ref{inherition} $(P_u,\,d_{P_u})=(P_u,\,b_{P_u})$ is a
maximal $d$-Brauer pair and the block $b_u$ is controlled by $N_C(P_u,\,d_{P_u})$. By Lemma \ref{hyperfocal subgroup of local subgroup},
the hyperfocal subgroup $Q_d$ of the block $d$ with respect to $(P_u,\,d_{P_u})$
is contained in $Q$.
If $Q_d=1$, then the block $d$ is nilpotent and $l(d)=e_u=1$.
If $Q_d\neq 1$, then $Q_d$ has to be isomorphic to
$\mathbb Z_{2^m}\times \mathbb Z_{2^m}$ for some $1\leq m\leq n$;
since $Q$ is contained in the center of $P$,
$Q_d$ is contained in the center of $P_u$;
by Lemmas \ref{inherition}, \ref{inertial subgroup of local subgroups} and \ref{ctrlb}, we have $l(d)=e_u=3$.
In a word, we have $l(d)=e_u$.
On the other hand,
let $(u,\, b^0_{\langle u\rangle})$ be
the $b_0$-Brauer element contained in the maximal $b_0$-Brauer pair $(P,\,b_P)$ and set $b^0_u=b^0_{\langle u\rangle}$.
By \cite[Corollary 3.6]{P01},
the Brauer categories of the block $b^0_u$ and $d$ are equivalent.
Then $e_u$ is also the inertial index of the block $b^0_u$.
By the structure theorem of blocks with normal defect group,
we have $l(b^0_u)=e_u$.
Therefore we have $l(b_u)=l(b^0_u)$ and $k(b)=k(b_0)$.
Similarly $k(c)=k(b_0)$.
\end{proof}

\begin{lem}\label{Ordinarynotctrl}
Assume that the block $b$ is not controlled by $N_G(P)$,
that $Q$ is contained in the center of $S$ and that the order of $Q$ is less than $16$. We have $k(b)=k(c)$.
\end{lem}

\begin{proof}
Since the block $b$ is not controlled by $N_G(P)$,
$\mathcal{F}_{(P,\,b_P)}(G,\,b)$ has only one essential object $(S,\,b_S)$ and
it is equal to $\mathcal{F}_{(P,\,b_P)}\big(N_G(S,\,b_S),\,b_S\big).$
For any $u\in P$, set $C=C_G(u)$ and denote by $f$ the block of $C$ such that
the pair $(\langle u\rangle,\,f)$ is the $b$-Brauer pair contained in $(P,\,b_P)$;
set $G'=N_G(S,\,b_S)$, $C'=C_{G'}(u)$ and $b'=(b_S)^{G'}$,
and denote by $f'$ the block of $C'$ such that $(\langle u\rangle,\,f')$ is the unique $b'$-Brauer pair contained in $(P,\,b_P)$.
For any $R\leq P$, set $R_u=C_R(u)$.
We assume that the pair $(\langle u\rangle,\,f)$ is fully normalized in
$\mathcal{F}_{(P,\,b_P)}(G,\,b)$.
Then the pair $(\langle u\rangle,\,f')$ is fully normalized in
$\mathcal{F}_{(P,\,b_P)}(G',\,b')$.
Denote by $(P_u,\,b'_{P_u})$ the $b'$-Brauer pair contained in
the $b'$-Brauer pair $(P,\,b_P)$.
For any subgroup $T$ of $P_u$,
denote by $(T,\,f_T)$ the $f$-Brauer pair contained in $(P_u,\,b_{P_u})$ and
by $(T,\,f'_T)$ the $f'$-Brauer pair contained in $(P_u,\,b'_{P_u})$.
By \cite[Corollary 3.6]{P01}, the correspondence $(T,\,f_T)\mapsto (T,\,f'_T)$ induces an equivalence between the Brauer categories
${\cal F}_{(P_u,\, b_{P_u})}(C,\, f)$ and ${\cal F}_{(P_u,\, b'_{P_u})}(C',\, f')$.

We are going to prove $l(f)=l(f')$.
By Lemma \ref{hyperfocal subgroup of local subgroup},
the hyperfocal subgroup $Q_f$ of the block $f$ with respect to
the maximal $f$-Brauer pair $(P_u,\,b_{P_u})$ is contained in $Q$.
Assume that $Q_f$ is of the form
$\mathbb Z_{2^i}\times \mathbb Z_{2^j}$ with $i\neq j$.
Then the block $f$ is nilpotent and $Q_f$ is equal to $1$.
That contradicts with $Q_f$ bigger than $1$.
Therefore $Q_f$ is of the form $\mathbb Z_{2^i}\times \mathbb Z_{2^i}$.
Assume that $i=0$.
Then $Q_f$ is equal to $1$,
the block $f$ is nilpotent and $l(f)=1$.
Since the Brauer categories ${\cal F}_{(P_u,\, b_{P_u})}(C,\, f)$ and
${\cal F}_{(P_u,\, b'_{P_u})}(C',\, f')$ are equivalent,
the block $f'$ is nilpotent and $l(f')=1=l(f)$.
Therefore in the sequel, we may assume that $i>0$.

Assume that $P_u\leq S$.
Since $P_u$ is contained in $S$,
by Proposition \ref{fusion system} the block $f$ is controlled by $N_C(P_u)$.
By Lemma \ref{ctrlb}, $l(f)$ is equal to $3$.
Since the Brauer categories ${\cal F}_{(P_u,\, b_{P_u})}(C,\, f)$ and
${\cal F}_{(P_u,\, b'_{P_u})}(C',\, f')$ are equivalent,
the block $f'$ is controlled by $N_{C'}(P_u)$ and $l(f')$ is equal to $3$ too.
Therefore we have $l(f)=l(f')$.

Now assume that $P_u$ is not contained in $S$.
By Proposition \ref{fusion system}, the block $f$ is not controlled by $N_C(P_u)$ and $l(f)$ is equal to $2$.
Since the Brauer categories ${\cal F}_{(P_u,\, b_{P_u})}(C,\, f)$ and
${\cal F}_{(P_u,\, b'_{P_u})}(C',\, f')$ are equivalent,
the block $f'$ is not controlled by $N_{C'}(P_u)$ and $l(f')$ is equal to $2$.
Therefore we have $l(f)=l(f')$.

Up to now, we already prove $l(f)=l(f')$ in general.
Then we have $k(b)=k(b')$.
Similarly, $k(c)=k(b')$.
Thus $k(b)=k(c)$.
\end{proof}

Now {\bf Theorem} \ref{MT} follows from Proposition \ref{fusion system} and
Lemmas \ref{ctrlb}, \ref{nctrlb}, \ref{Ordinaryctrl} and \ref{Ordinarynotctrl}.
\vspace{0.5cm}

As an application of Theorem \ref{MT},
we show that Alperin Weight Conjecture (see \cite{A87}) holds in our setting.
The following consequence is more or less well-known.

\begin{lem}\label{weight}
Assume that $R$ is a normal subgroup of $P$ and
that the block $b$ is controlled by $N_G(R,\,b_R)$.
Then for any weight $(T,\,W)$ of the block $b$,
we have $R\subseteq T$.
\end{lem}

\begin{proof}
  Let $(T,\,W)$ be a weight for the block $b$,
  namely, $W$ is a simple $kN_G(T,\,b_T)b_T$-module
  with the vertex $T$.
  Since the block $b$ is controlled by $N_G(R,\,b_R)$,
  $N_G(T,\,b_T)=C_G(T)\big(N_G(R,\,b_R)\cap N_G(T,\,b_T)\big)$.
  If $R\nleq T$,
  then $T\lneqq N_{TR}(T,\,b_T)$.
  Set $X=N_{TR}(T,\,b_T)$.
  Since the vertex of $W$ is $T$,
  the defect of the block $b_T$ is $Z(T)$
  and therefore the block $b_T$ is nilpotent.
  Hence the block $(b_T)^{XC_G(T)}$ is also nilpotent.
  Since $N_G(T,\,b_T)=C_G(T)\big(N_G(R,\,b_R)\cap N_G(T,\,b_T)\big)$,
  $XC_G(T)\unlhd N_G(T,\,b_T)$.
  Set $L=XC_G(T)$.
  Let $U$ be the unique simple $kLb_Q$-module.
  Then $X$ has to be a vertex of $U$
  and $U$ is the unique direct summand of $W_L$.
  But the module $W$ has $T$ as the vertex.
  So $T$ has to contain the vertex of $U$.
  Hence $X=T$.
  It is a contradiction.
\end{proof}

\begin{cor}
Alperin Weight Conjecture holds for the block $b$
when one of the following conditions is satisfied:

\noindent{\bf (i)} the block $b$ is controlled by $N_G(P)$ and
$Q$ is contained in the center of $P$;

\noindent{\bf (ii)} the block $b$ is not controlled by $N_G(P)$,
$Q$ is contained in the center of $S$,
and $|Q|$ is at most $16$.
\end{cor}

\begin{proof}
In case (i), the corollary follows from Lemma \ref{weight} and Theorem \ref{MT}.
Now we assume case (ii).
Notice in this case we have $N_G(P,\,b_P)=PC_G(P)$
(see the proof of Proposition \ref{fusion system}).
By Lemma \ref{weight} and Theorem \ref{MT},
it suffices to show that there is only one simple $kN_G(S,\,b_S)$-module
in the block $(b_S)^{N_G(S,\,b_S)}$ with vertex $S$, up to isomorphism.
By Lemma \ref{nctrlb}, $l\big((b_S)^{N_G(S,\,b_S)}\big)=2$.
Notice that $N_G(S,\,b_S)/SC_G(S)\cong {\rm S}_3$.
Let $A$ be the subgroup of $N_G(S,\,b_S)$ such that
$SC_G(S)\leq A$ and $A/SC_G(S)\cong \mathbb Z_3$.
Since the block $b_S$ of $SC_G(S)$ is nilpotent,
the block $b_S$ of $A$ has $3$ simple modules, up to isomorphism,
which have the same vertex $S$.
Since $A\unlhd N_G(S,\,b_S)$ and $|N_G(S,\,b_S):A|=2$,
by Clifford's Theorem,
there exists a simple $kA$-module $V$ such that $V^{N_G(S,\,b_S)}$ is the unique simple $kN_G(S,\,b_S)$-module belonging to the block $(b_S)^{N_G(S,\,b_S)}$
with vertex $S$, up to isomorphism.
We are done.
\end{proof}

\end{document}